\documentclass[a4paper,10pt]{article}

\usepackage{amssymb} 
\usepackage{amsfonts}
\usepackage{amsmath}
\usepackage{amsthm}
\usepackage[utf8]{inputenc}
\usepackage{graphicx}
\title{{\bfseries A real sextic surface with $45$ handles}}
\author{\textbf{Arthur Renaudineau}}
\date{}

\newtheorem{theorem}{Theorem}

\newtheorem{remark}{Remark}

\newtheorem{definition}{Definition}
\newtheorem*{conjecture}{Conjecture}
\newtheorem{proposition}{Proposition}
\newtheorem{question}{Question}
\newcommand{\R}{\mathbb{R}}

\newcommand{\Z}{\mathbb{Z}}
\newcommand{\CP}{\mathbb{CP}}

\newcommand{\PP}{\mathbb{P}}
\newcommand{\C}{\mathbb{C}}

\begin{document}
\maketitle
\begin{abstract}
It follows from classical restrictions on the topology of real algebraic varieties that the first Betti number of the real part of a real nonsingular sextic in $\CP^3$ can not exceed $94$. We construct a real nonsingular sextic $X$ in $\CP^3$ satisfying $b_1(\R X)=90$, improving a result of F.Bihan. The construction uses Viro's patchworking and an equivariant version of a deformation due to E.Horikawa.  
\end{abstract}

\section{Introduction}
A \textit{real algebraic variety} is a complex algebraic variety $X$ equipped with an antiholomorphic involution $c:X\rightarrow X$. Such an antiholomorphic involution is called a \textit{real structure} on $X$. The \textit{real part} of $(X,c)$, denoted by $\R X$, is the set of points fixed by $c$. 
The \textit{standard real structure} $c_0$ on $(\C^*)^n$ is defined by 
$$
c_0(Z_1,...,Z_n)=(\overline{Z_1},...,\overline{Z_n}).
$$
The \textit{standard real structure} on a toric variety of dimension $n$ is the real structure induced by the standard real structure on $(\C^*)^n$.
In this text, the only real structures we consider  on toric varieties are standard real structures. 
A \textit{real subvariety} of a toric variety is a subvariety stable by the standard real structure.
For example, a real algebraic surface in $\CP^3$ is the zero set of a real homogeneous polynomial in $4$ variables. Unless otherwise specified, all varieties considered are nonsingular. The homology is always considered with $\Z/2\Z$-coefficients. For a topological space $A$, we put $b_i(A)=\dim_{\Z/2\Z}H_i(A,\Z/2\Z)$. The numbers $b_i(A)$ are called \textit{Betti numbers $($\textit{with $\Z/2\Z$ coefficients}$)$ of $A$}. All polytopes considered are convex lattice polytopes in $\R^n$. 
\\ Let us remind several classical inequalities and congruences in topology of real algebraic varieties.
\\\\\textbf{Smith-Thom inequality and congruence:}
Let $X$ be a compact real algebraic variety. Then
$$
b_*(\R X)\leq b_*(X) \mbox{ and } b_*(\R X)\equiv b_*(X) \mod 2,
$$
where $b_*$ is the sum of all Betti numbers. The variety $X$ is called an \textit{$M$-variety} if $b_*(\R X)= b_*(X)$ and an \textit{$(M-a)$-variety} if $b_*(\R X)= b_*(X)-2a$.
\\\\\textbf{Petrovsky-Oleinik inequalities:}
Let $X$ be a compact complex Kähler manifold of real dimension $4n$ equipped with a real structure. Then
$$
2-h^{n,n}(X)\leq\chi(\R X)\leq h^{n,n}(X),
$$
where $\chi$ denotes the Euler caracteristic and $h^{p,q}$ denotes the $(p,q)$-Hodge number.
\\\\\textbf{Rokhlin congruence:}
Let $X$ be a compact $M$-variety of real dimension $4n$. Then
$$
\chi(\R X)\equiv\sigma(X) \mod 16,
$$ 
where $\sigma(X)$ is the signature of $X$.
\\\\\textbf{Gudkov-Kharlamov congruence:}
Let $X$ be a compact ($M-1$)-variety of real dimension $4n$. Then
$$
\chi(\R X)\equiv\sigma(X)\pm 2 \mod 16.
$$ 
For an introduction concerning restrictions on the topology of real algebraic varieties, see ~\cite{Wilson} or ~\cite{DegKhar}.
\\Let $X$ be a compact connected simply-connected projective real surface. From the Smith-Thom inequality and the Petrovsky-Oleinik inequalities, one can deduce bounds for $b_0(\R X)$ and $b_1(\R X)$ in terms of Hodge numbers of $X$:
\begin{equation}
b_0(\R X)\leq \frac{1}{2}(h^{2,0}(X)+h^{1,1}(X)+1),
\label{equation1}
\end{equation}
\begin{equation}
b_1(\R X)\leq h^{2,0}(X)+h^{1,1}(X).
\label{equation2}
\end{equation}
These bounds are not sharp in general. One can then ask the following questions.
\begin{question}
What is the maximal possible value of $b_0(\R X)$ for a real algebraic surface $X$ in $\CP^3$ of a given degree?
\end{question}
\begin{question}
What is the maximal possible value of $b_1(\R X)$ for a real algebraic surface $X$ in $\CP^3$ of a given degree?
\label{question}
\end{question}
\noindent If the degree is greater than $4$, these questions are still widely open.
In 1980, O.Viro formulated the following conjecture.
\begin{conjecture}$($O.Viro$)$
\\Let $X$ be a compact connected simply-connected projective real surface. Then
$$
b_1(\R X)\leq h^{1,1}(X).
$$
\end{conjecture}
This conjecture was an attempt to give an answer to Question \ref{question}. When $X$ is the double covering of $\CP^2$ ramified along a curve of an even degree, this conjecture is a reformulation of Ragsdale's conjecture (see ~\cite{Viro80}). The first counterexample to Ragsdale's conjecture was constructed by I.Itenberg (see ~\cite{Itenberg93}) using Viro's combinatorial patchworking (see Section \ref{Viro} or ~\cite{Itenberg1997} or ~\cite{Bihan2001}). This first counterexample opened the way to various counterexamples to Viro's conjecture and constructions of real algebraic surfaces with many connected components (see ~\cite{Itenberg1997}, ~\cite{Bihan2001}, ~\cite{Bihan2}, ~\cite{Brugalle2006},  ~\cite{ItKhar} and ~\cite{Orevkov2001}). It is not known whether Viro's conjecture is true for $M$-surfaces.
\\
In the case where $X$ is a real algebraic surface of degree $d$ in $\CP^3$, inequalities (\ref{equation1}) and (\ref{equation2}) specialize to the following ones:
\begin{equation}
b_0(\R X)\leq\frac{5}{12}d^3-\frac{3}{2}d^2+\frac{25}{12}d, \tag{1'}
\end{equation}
\begin{equation}
b_1(\R X)\leq\frac{5}{6}d^3-3d^2+\frac{25}{6}d-1\tag{2'}.
\label{equation3}
\end{equation}
Consider the case where $X$ has degree $6$. Then $h^{1,1}(X)=86$, and the inequality (\ref{equation3}), combined with Petrovsky-Oleinik inequalities and Rokhlin congruence, gives $b_1(\R X)\leq 94$. F.Bihan constructed in ~\cite{Bihan2001}, using Viro's combinatorial patchworking, a real sextic $X$ satisfying $b_1(\R X)=88$. Moreover, the real part of $X$ is homeomorphic to $6S\amalg S_2\amalg S_{42}$, where $S$ denotes a $2$-dimensional sphere and $aS_\alpha$ denotes the disjoint union of $a$ spheres each having $\alpha$ handles. In this note, we improve this construction. 
\begin{theorem}
There exists a real sextic surface $X$ in $\CP^3$ satisfying $b_1(\R X)=90$ such that $$\R X\simeq 4S\amalg 2S_2\amalg  S_{41}.$$
\label{theorem1}
\end{theorem}
The paper is organized as follows. In Sections \ref{Viro} and \ref{sectionT}, we remind Viro's method and some results about the Euler caracteristic of T-surfaces. In Section \ref{sectionequi}, we describe a class of real algebraic surfaces, the so-called surfaces of type $(1c)$, and an equivariant deformation of a real surface of type $(1c)$ to a real sextic surface.  In Section \ref{sectionfin}, we use Viro's combinatorial patchworking to construct a real surface $Z$ of type $(1c)$. Then, using the general Viro's method, we slightly modify the construction of $Z$ to obtain a real surface $Y$ of type $(1c)$ satisfying $$\R Y\simeq 4S\amalg 2S_2\amalg  S_{41}.$$
The existence of a real sextic surface $X$ satisfying $92\leq b_1(\R X)\leq 94$ is still unknown.
\\
\\
\textbf{Acknowledgments.} I am very grateful to Erwan Brugallé and Ilia Itenberg for useful discussions and advisements.

\section{Viro's method}
\label{Viro}
\subsection{T-construction}
The combinatorial patchworking construction (or T-construction) works in any dimension.
\\Let $(u_1,...,u_n)$ be coordinates in $\R^n$, and let $\Delta$ be a $n$-dimensional polytope in $\R_+^n$, where $\R_+=\lbrace x\in\R\mbox{ }\vert\mbox{ } x\geq 0\rbrace$. Denote by $\mathit{Tor}(\Delta)$ the toric variety associated with $\Delta$. We denote by $\R\mathit{Tor}(\Delta)$ the real part of $\mathit{Tor}(\Delta)$ for the standard real structure. Take a triangulation $\tau$ of $\Delta$ with vertices having integer coordinates, and a distribution of signs at the vertices of $\tau$. Denote the sign at any vertex $(i_1,...,i_n)$ by $\delta_{i_1,...,i_n}$. For $\epsilon=(\epsilon_1,...,\epsilon_n)\in(\Z/2\Z)^n$, let $s_\epsilon$ be the symmetry of $\R^n$ defined by
$$
s_\epsilon(u_1,...,u_n)=((-1)^{\epsilon_1}u_1,...,(-1)^{\epsilon_n}u_n).
$$
Denote by $\Delta_{*}$ the union
$$
\cup_{\epsilon\in(\Z/2\Z)^n}s_\epsilon(\Delta).
$$
Extend the triangulation $\tau$ to a symmetric triangulation of $\Delta_*$, and the distribution of signs $\delta_{i_1,...,i_n}$ to a distribution at the vertices of the extended triangulation using the following formula: 
$$
\delta_{s_\epsilon (i_1,...,i_n)}=\left(\prod_{j=1}^{j=n}(-1)^{\epsilon_j i_j}\right)\delta_{i_1,...,i_n}.
$$
\\If a tetrahedron $T$ of the triangulation of $\Delta_*$ has vertices of different signs, denote by $S_T$ the convex hull of the middle points of the edges of $T$ having endpoints of opposite signs. Denote by $S$ the union of all such $S_T$. It is a $(n-1)$ piecewise-linear manifold contained in $\Delta_*$. If $\Gamma$ is a face of $\Delta_*$, then, for all integer vectors $\alpha$ orthogonal to $\Gamma$, identify $\Gamma$ with $s_\alpha(\Gamma)$. Denote by $\hat{\Delta}$ the quotient of $\Delta_*$ under these identifications, and by $\pi_\Delta$ the quotient map. The real part  $\R\mathit{Tor}(\Delta)$ is homeomorphic to $\hat{\Delta}$. 
\\The triangulation $\tau$ of $\Delta$ is said to be \textit{convex} if there exists a convex piecewise-linear function $\nu:\Delta\rightarrow\R$ whose domains of linearity coincide with the tetrahedra of $\tau$. 

\begin{theorem}$($O.Viro$)$
\label{Patch}
\\Assume that the only singularities of $\mathit{Tor}(\Delta)$ correspond to the vertices of $\Delta$ and that the triangulation $\tau$ of $\Delta$ is convex. Then there exists a nonsingular real algebraic hypersurface $X$ in $\mathit{Tor}(\Delta)$ belonging to the linear system associated with $\Delta$, and a homeomorphism $\R\mathit{Tor}(\Delta)\rightarrow\hat{\Delta}$ mapping $\R X$ to $\pi_\Delta(S)$.
\end{theorem}

A polynomial defining such an hypersurface $X$ can be written down explicitly. If $t>0$ is sufficiently small, the polynomial
\begin{equation}
\sum_{(i_1,...,i_n)\in V}\delta_{i_1,...,i_n}\prod_{j=1}^{j=n}\left(x_j^{i_j}\right)
t^{\nu(i_1,...,i_n)}
\label{Viropoly}
\end{equation}
(where V is the set of vertices of $\tau$ and $\nu$ is a function ensuring the convexity of $\tau$) defines an hypersurface in $(\C^*)^n$, such that the compactification of this hypersurface in $\mathit{Tor}(\Delta)$ has the properties described in Theorem \ref{Patch}.
\begin{definition}
A polynomial of the form (\ref{Viropoly}) is called a \textit{Viro polynomial} and an hypersurface defined by such a polynomial $($for sufficiently small $t>0)$ is called a T-hypersurface.
\end{definition}
\begin{remark}
The assumption on the singularities of $\mathit{Tor}(\Delta)$ is not essential. See Section \ref{GeneralViro}.
\end{remark}
The T-construction is a particular case of a more general construction, called Viro's patchworking or Viro's method. 
\subsection{General Viro's method}
\label{GeneralViro}
In this construction, we glue together more complicated pieces than before. These pieces are called \textit{charts of polynomials}.
\begin{definition}
Let $f$ be a polynomial in $\R[x_1,...,x_n]$ and $Z(f)$ be the set \linebreak $\lbrace x\in(\R^*)^n\vert\mbox{ } f(x)=0\rbrace$. Let $\Delta(f)\subset (\R_+)^n$ be the Newton polygon of $f$. In the octant $(\R_+^*)^n$, we define $\phi$ as 
$$
\begin{array}{ccccc}
\phi & : & (\R_+^*)^n & \to & (\R^*)^n\\
 & & z & \mapsto & \dfrac{\sum_{i\in\Delta\cap\Z^n}\mid z^{i}\mid i}{\sum_{i\in\Delta\cap\Z^n}\mid z^{i}\mid}.  \\
\end{array}
$$
In the octant $s_\epsilon((\R_+^*)^n)$, we put 
$$
\phi(s_\epsilon(z))=s_\epsilon(\phi(z)),
$$
where $s_\epsilon(x_1,...,x_n)=((-1)^{\epsilon_1} x_1,...,(-1)^{\epsilon_n} x_n)$.
\\We call chart of $f$ the closure of $\phi(Z(f))$ in $\Delta(f)_*$. Denote by $C(f)$ the chart of $f$.
\end{definition}
\begin{definition}
Let $f=\sum a_ix^i$ be a polynomial in $n$ variables.
Let $\Gamma\subset\Z^n$ be a subset of the Newton polygon $\Delta(f)$ of $f$. The \textit{truncation} of $f$ to $\Gamma$ is the polynomial $f^\Gamma$ defined by 
$
f^\Gamma=\sum_{i\in\Gamma} a_ix^i.
$
\end{definition}
\begin{definition}
A polynomial $f$ is called non-degenerated with respect to its Newton polygon $\Delta(f)$ if for any face $\Gamma$ of $\Delta(f)$ $($including ${\Delta(f)}$ itself $)$, the polynomial $f^\Gamma$ defines a nonsingular hypersurface in $(\C^*)^k$, where $k$ is the dimension of $\Gamma$.
\end{definition}
Let $\Delta$ be an $n$-dimensionnal polytope in $\R_+^n$ and let $\cup_{i\in I}\Delta_i$ be a decomposition of $\Delta$ such that all the $\Delta_i$ have vertices with integer coordinates. For any $i\in I$, take a polynomial $f_i$ such that the $f_i$'s verify the following properties:
\begin{itemize}
\item for all $i\in I$, the Newton polygon of $f_i$ is $\Delta_i$,
\item if $\Gamma=\Delta_i\cap\Delta_j$, then $f_i^\Gamma=f_j^\Gamma$,
\item for all $i\in I$, the polynomial $f_i$ is non-degenerated with respect to $\Delta_i$.
\end{itemize}
The polynomials $f_i$ define an unique polynomial $f=\sum_{w\in\Delta\cap\Z^n}a_wx^w$, such that $f^{\Delta_i}=f_i$ for all $i\in I$. The decomposition $\cup_{i\in I}\Delta_i$ of $\Delta$ is said to be \textit{convex} if there exists a convex piecewise-linear function $\nu:\Delta\rightarrow\R$ whose domains of linearity coincide with the $\Delta_i$.
\begin{theorem}$($O.Viro$)$
\label{GeneralPatch}
\\Assume that the decomposition $\cup_{i\in I}\Delta_i$ of $\Delta$  is convex and let $\nu:\Delta\rightarrow\R$ be a function certifying its convexity. Define the associated Viro polyomial \linebreak $f_t=\sum_{w\in\Delta\cap\Z^n}a_wt^{\nu(w)}x^w$. Then there exists $t_0>0$ such that if $0<t<t_0$, then $f_t$ is non-degenerated with respect to $\Delta$ and there exists an homeomorphism of $\hat{\Delta}$ sending $\pi_{\Delta(f)}(C(f_t))$ to $\pi_{\Delta(f)}(\cup_{i\in I} C(f_i))$.
\end{theorem}

For more details about the general Viro's method, see for example ~\cite{Viro1983} or ~\cite{Risler1992}.

\section[Euler caracteristic of the real part of a T-surface]{Euler caracteristic of the real part of a \linebreak T-surface}
\label{sectionT}

We remind in this section some results about the topology of T-surfaces. Let us introduce first some terminology concerning simplices and triangulations of polytopes.
\begin{definition}
The integer volume of an $n$-dimensional simplex in $\R^n$ is equal to $n!$ times its euclidean volume. An $n$-dimensional simplex in $\R^n$ is called maximal if it does not contain other integer points than its vertices. A maximal simplex is called primitive if its integer volume is equal to $1$ and elementary if its integer volume is odd.
\end{definition}
\begin{definition}
A triangulation of an $n$-dimensional polytope $P$ in $\R^n$ is called maximal $($resp., primitive$)$ if all $n$-dimensional simplices in the triangulation are maximal $($resp., primitive$)$.
\end{definition}
\begin{definition}
The star of a face $F$ in a triangulation $\tau$, denoted by $st(F)$, is the union of all simplices in $\tau$ having $F$ as face.
\end{definition}
\begin{definition}
We say that an edge $\lambda$ of a triangulation $\tau$ is of length $n$ if $\lambda$ contains $n+1$ integer points.
\end{definition}
\begin{definition}
\label{definition}
Let $\tau$ be a triangulation containing an edge $\lambda$ of length $2$. Suppose that $\lambda$ is the only edge of length greater than $1$ in $st(\lambda)$. The refined triangulation is obtained by adding the middle point of $\lambda$ to the set of vertices of $\tau$ and by subdividing each tetrahedron in $st(\lambda)$ accordingly.
\end{definition}
Let $\Delta$ be a $3$-dimensional polytope in $\R_+^3$. Suppose that the only singularities of $\mathit{Tor}(\Delta)$ correspond to the vertices of $\Delta$. The real part of a T-surface in $\mathit{Tor}(\Delta)$ admits a cellular decomposition coming from the triangulation of $\hat{\Delta}$. This cellular decomposition allows one to compute the Euler caracteristic of the real part.
\begin{proposition}$($see ~\cite{Bihan2001} $)$
\\Suppose that $\Delta$ admits a maximal triangulation $\tau$. Given a distribution of signs $D(\tau)$, denote by $N$ $($resp., $P)$ the set of tetrahedra of even volume in $\tau$ with negative $($resp., positive$)$ product of signs at the vertices. Let $E$ be the set of elementary tetrahedra in $\tau$. Let $Z$ be a T-surface obtained from $(\tau,D(\tau))$. Then
$$
\chi(\R Z)=\sigma(\C Z)+ \sum_{T \mbox{ tetrahedra in }\tau}(Vol(T)-\varepsilon_T),
$$
where $\varepsilon_T=0,1,2$ if $T\in N,E,P$ respectively.
\label{prop2}
\end{proposition}

\begin{proposition}$($see ~\cite{Bihan2001} $)$
\\Suppose that $\Delta$ admits a triangulation $\tau$ with an edge $\lambda$ of length $2$ $($with middle point $a)$ such that $\lambda$ is the only edge of length greater than $1$ in $st(\lambda)$. Denote by $k$ the dimension of the minimal face of $\Delta$ containing $\lambda$. Denote by  $\tau_a$ the refined triangulation $($see Definition \ref{definition}$)$. Let $D(\tau)$ be any distribution of signs in $\tau$ and extend it to $D(\tau_a)$ choosing any sign of $a$. Let $P_a$ be the set of tetrahedra in $st(a)$ which are of even volume and positive product of signs at the vertices.  Let $E_a$ be the set of elemetary tetrahedra in $st(a)$. Denote by $Z$, resp. $Z_a$, a T-surface obtain from $(\tau,D(\tau))$, resp. $(\tau_a,D(\tau_a))$.
\\If the endpoints of $\lambda$ have opposite signs, then $\chi(\R Z)=\chi(\R Z_a)$, and
$$
\chi(\R Z)-\chi(\R Z_a)=\#(E_a)+2\#(P_a)-2^k,
$$ 
otherwise.
\label{prop3}
\end{proposition}

\section{An equivariant deformation}
\label{sectionequi}
In his construction, Bihan used an equivariant version of Horikawa's deformation of surfaces of type $(1c)$ in $\CP^4(2)$ (see ~\cite{Horikawa1993}).
\begin{definition}
A \textbf{family of compact complex surfaces} $\mathcal{F}=(L,p,B)$ consists of a pair of connected complex manifolds $L$ and $B$, and a proper holomorphic map $p:L\rightarrow B$ which is a submersion and whose fibers $L_b$ are connected surfaces.
\\Let $V$ be a connected compact complex surface. An \textbf{elementary deformation} of $V$ parametrised by a complex contractible manifold $B$ consists of a connected complex manifold $L$, a base point $b_0\in B$, a family $\mathcal{F}=(L,p,B)$ and an injective morphism $i:V\rightarrow L$ such that $i(V)=L_{b_0}$.
\\A \textbf{result of an elementary deformation of $V$} is a connected complex surface which is a fiber of the map $p$.
\\On the set of complex surfaces, introduce the equivalence relation generated by elementary deformations and isomorphisms. Any surface belonging to the equivalent class of $V$ is called a \textbf{deformation} of $V$.
\\Suppose that $(V,c)$ is a real surface. An \textbf{elementary equivariant deformation} of $(V,c)$ is an elementary deformation of $V$ such that $L$ (resp., $B$) is equipped with an antiholomorphic involution $Conj:L\rightarrow L$ (resp., ${conj:B\rightarrow B}$) satisfying $p\circ Conj=conj\circ p$, $conj(b_0)=b_0$ and $Conj\circ i=i\circ c$.
\\On the set of real surfaces, introduce the equivalence relation generated by elementary equivariant deformations and real isomorphisms.
\end{definition}
Consider the $4$-dimensional weighted projective space $\CP^4(2)$ with complex homogeneous coordinates $Z_0,Z_1,Z_2,Z_3$ of weight $1$ and $Z_4$ of weight $2$.
\begin{definition}$($see ~\cite{Horikawa1993} $)$
\\An algebraic surface $Y$ in $\CP^4(2)$ is said to be of type $(1c)$ if $Y$ is defined by the following system of equations:
\label{definition1c}
$$\left\lbrace
\begin{array}{l}
Z_4^3+f_2(Z)Z_4^2+f_4(Z)Z_4+f_6(Z)=0,\\
Z_0Z_3-Z_1Z_2=0.\\
\end{array}\right.$$ 
where $f_{2i}(Z)$ is a homogeneous polynomial of degree $2i$ in the variables \linebreak $Z_0,Z_1,Z_2,Z_3$. 
\\We define a real algebraic surface of type $(1c)$ to be a complex algebraic surface of type $(1c)$ invariant under the standard real structure on $\CP^4(2)$. 
\end{definition}
In ~\cite{Horikawa1993}, Horikawa showed that any nonsingular algebraic surface of type $(1c)$ can be deformed to a nonsingular surface of degree $6$ in $\CP^3$. The same result is true in the real category.
\begin{proposition} $($see ~\cite{Bihan2} $)$
\label{propositionHor} 
\\Let $Y$ be a nonsingular real algebraic surface of type $(1c)$. Then, there exists an equivariant deformation of $Y$ to a nonsingular real surface $X$ of degree $6$ in $\CP^3$.
\label{propHor}
\end{proposition}
\begin{proof}
Consider the elementary equivariant deformation of $Y=Y_0$ determined by the family $(Y_\epsilon)$ for $\epsilon\in\R$, where $Y_\epsilon$ is defined by the following system of equations:
$$\left\lbrace
\begin{array}{l}
Z_4^3+f_2(Z)Z_4^2+f_4(Z)Z_4+f_6(Z)=0,\\
Z_0Z_3-Z_1Z_2-\epsilon Z_4=0.\\
\end{array}\right.$$ 
As $Y$ is a nonsingular surface, then for sufficiently small $\epsilon$, the surface $Y_\epsilon$ is nonsingular.
The system defining the surface $Y_\epsilon$ can be transformed into:
$$\left\lbrace
\begin{array}{l}
(\frac{Z_0Z_3-Z_1Z_2}{\epsilon})^3+f_2(Z)(\frac{Z_0Z_3-Z_1Z_2}{\epsilon})^2+f_4(Z)(\frac{Z_0Z_3-Z_1Z_2}{\epsilon})+f_6(Z)=0,\\
\mbox{}\\
Z_4=\frac{Z_0Z_3-Z_1Z_2}{\epsilon}.\\
\end{array}\right.$$ 
Now, consider the projection
$$
\begin{array}{ccccc}
p & : & \CP^4(2)\setminus\lbrace(0:0:0:0:1)\rbrace & \to & \CP^3\\
 & & (Z_0:Z_1:Z_2:Z_3:Z_4) & \mapsto & (Z_0:Z_1:Z_2:Z_3).  \\
\end{array}
$$
The point $(0:0:0:0:1)\in\CP^4(2)$ does not belong to $Y_\epsilon$, hence $p\vert_{Y_\epsilon}$ is well defined. The projection $p$ produces a complex isomorphism between $Y_\epsilon$ and the algebraic surface $X_\epsilon$ of degree $6$ in $\CP^3$ defined by the polynomial
$$
(\frac{Z_0Z_3-Z_1Z_2}{\epsilon})^3+f_2(Z)(\frac{Z_0Z_3-Z_1Z_2}{\epsilon})^2+f_4(Z)(\frac{Z_0Z_3-Z_1Z_2}{\epsilon})+f_6(Z)=0.
$$
Moreover, this isomorphism is equivariant with respect to the involution $c$ and the standard involution on $\CP^3$.
\end{proof}
\begin{remark}
This deformation can be geometrically understood as a deformation of $\CP^3$ to the normal cone of a nonsingular quadric. $($See ~\cite{Fulton1984} for the general process of deforming an algebraic variety to the normal cone of a subvariety$)$.
\end{remark}
\begin{remark}
Any surface of type $(1c)$ is a hypersurface in the quadric defined by the equation $(Z_0Z_3-Z_1Z_2=0)$ in $\CP^4(2)$. This quadric is a projective toric variety. In particular, one may use Viro's patchworking to produce real surfaces in $Q$. A natural polytope which may be used to apply Viro's patchworking to produce real algebraic surfaces of type $(1c)$ is the polytope $Q$ with vertices $(0,0,0),(6,0,0),(6,6,0),(0,6,0),(0,0,3)$ in $\R^3$ $($see Figure \ref{polytope Q}$)$.
\end{remark}
\begin{figure}[!h]
\includegraphics[width=18cm,height=8cm]{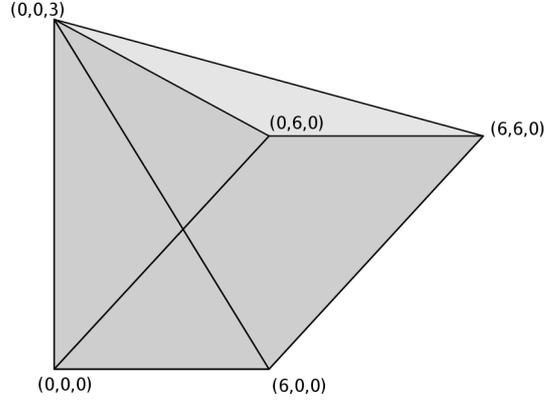}
\caption{Polytope $Q$.}
\setlength\abovecaptionskip{+1cm}
\label{polytope Q}
\end{figure}

\section{Construction of a surface $X$ of degree $6$ with $45$ handles}
\label{sectionfin}

\begin{proposition}
There exists a real algebraic surface $Y$ of type $(1c)$ such that
$$
\R Y\simeq 4S\amalg 2S_2\amalg  S_{41}.
$$
\label{prop4}
\end{proposition} 
\begin{proof}[Proof of Theorem \ref{theorem1}]
Performing the equivariant deformation described in Proposition \ref{propHor} to the surface $Y$, we obtain a real sextic surface $X$ in $\CP^3$, such that
$$
\R X\simeq 4S\amalg 2S_2\amalg  S_{41}.
$$
\end{proof}
The rest of the article is devoted to the proof of Proposition \ref{prop4}. Our strategy is first to describe a T-construction of an auxilliary surface $Z$ of Newton polytope $Q$. Then, we use the general Viro's patchworking method to modify slightly the construction.
\subsection{The auxilliary surface $Z$ \label{aux}}
We describe a triangulation $\tau$ of $Q$ and a distribution of signs $D(\tau)$ at the vertices of $\tau$.
Consider the cone $C$ with vertex $(1,0,2)$ over the square $Q_0=Q\cap\lbrace w=0\rbrace$ (see Figure \ref{cone C}).
\begin{figure}[h!]
\includegraphics[width=16cm,height=7cm]{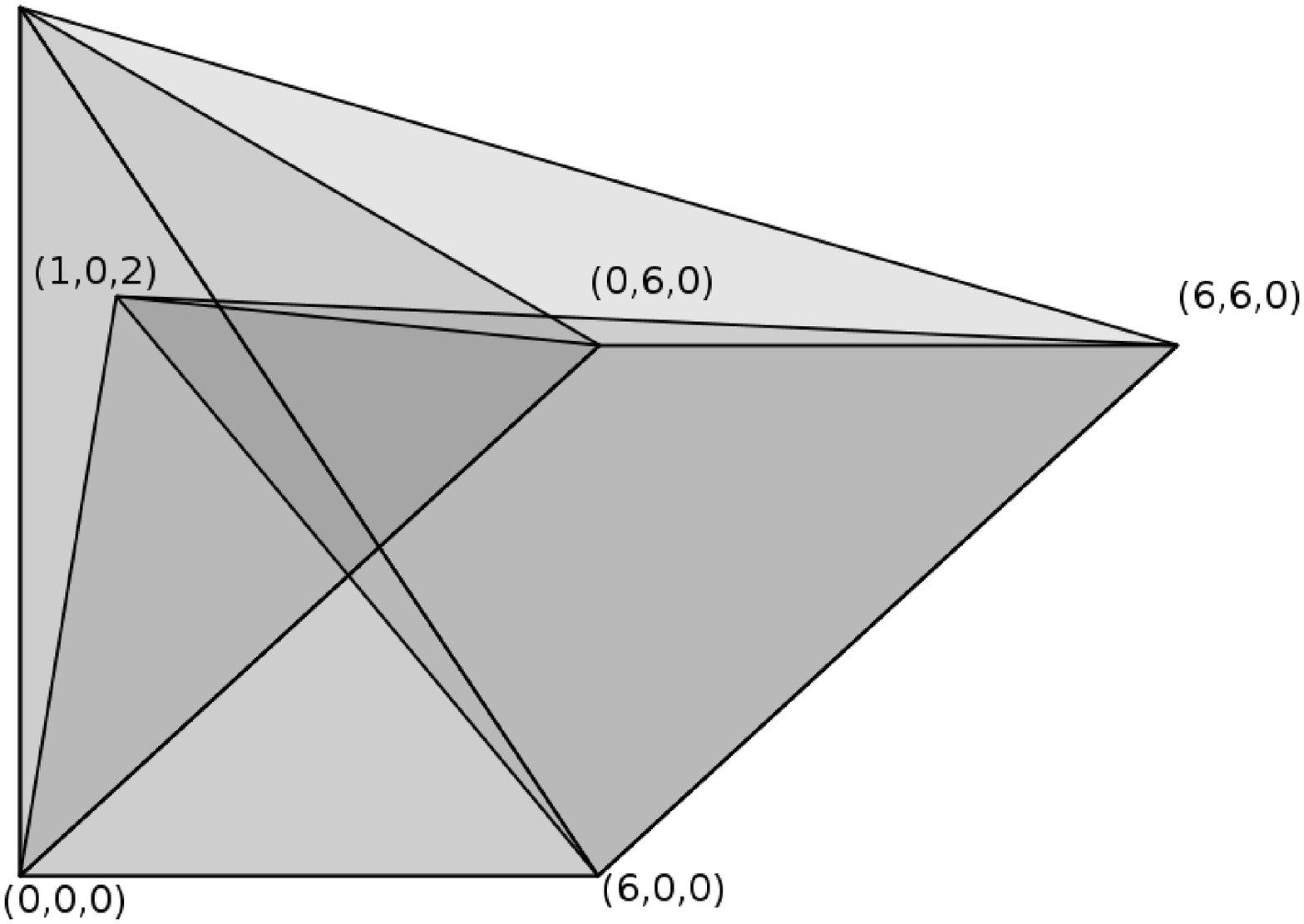}
\setlength\abovecaptionskip{-1cm}
\caption{Cone $C$.}
\label{cone C}
\end{figure}
Take any primitive convex triangulation of $Q_0$ containing the edges depicted in Figure \ref{courbeaux}.
\begin{figure}[h!]
\centerline{
\includegraphics[width=15cm,height=7cm]{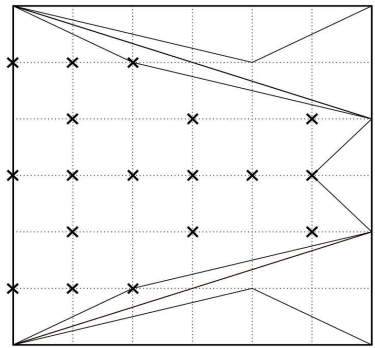}}

\setlength\abovecaptionskip{-1cm}
\caption{The fixed part of a triangulation of $Q_0$ and the distribution of signs. A point gets a sign $+$ if and only if it is ticked.}
\label{courbeaux}
\end{figure}
Then, triangulate $C$ into the cones with vertex $(1,0,2)$ over the triangles of the triangulation of $Q_0$. The triangulation of the cone $C$ contains $12$ edges of length $2$ (edges joining $(1,0,2)$ to the points of coordinates $(1,0)\mod 2$ inside $Q_0$). For the three edges $[(1,0,2)-(1,0,0)]$, $[(1,0,2)-(3,0,0)]$ and $[(1,0,2)-(5,0,0)]$ of length $2$, refine the triangulation as explained in Definition \ref{definition}.
\\Consider the tetrahedra $\alpha_1$ and $\alpha_2$ with vertices $(1,0,2),(6,6,0),(4,0,1),(6,0,0)$ and $(1,0,2),(0,6,0),(0,0,1),(0,0,0)$ respectively. See Figure \ref{alpha1} for a picture of $\alpha_1$.
\begin{figure}[h!]
\includegraphics[width=18cm,height=8cm]{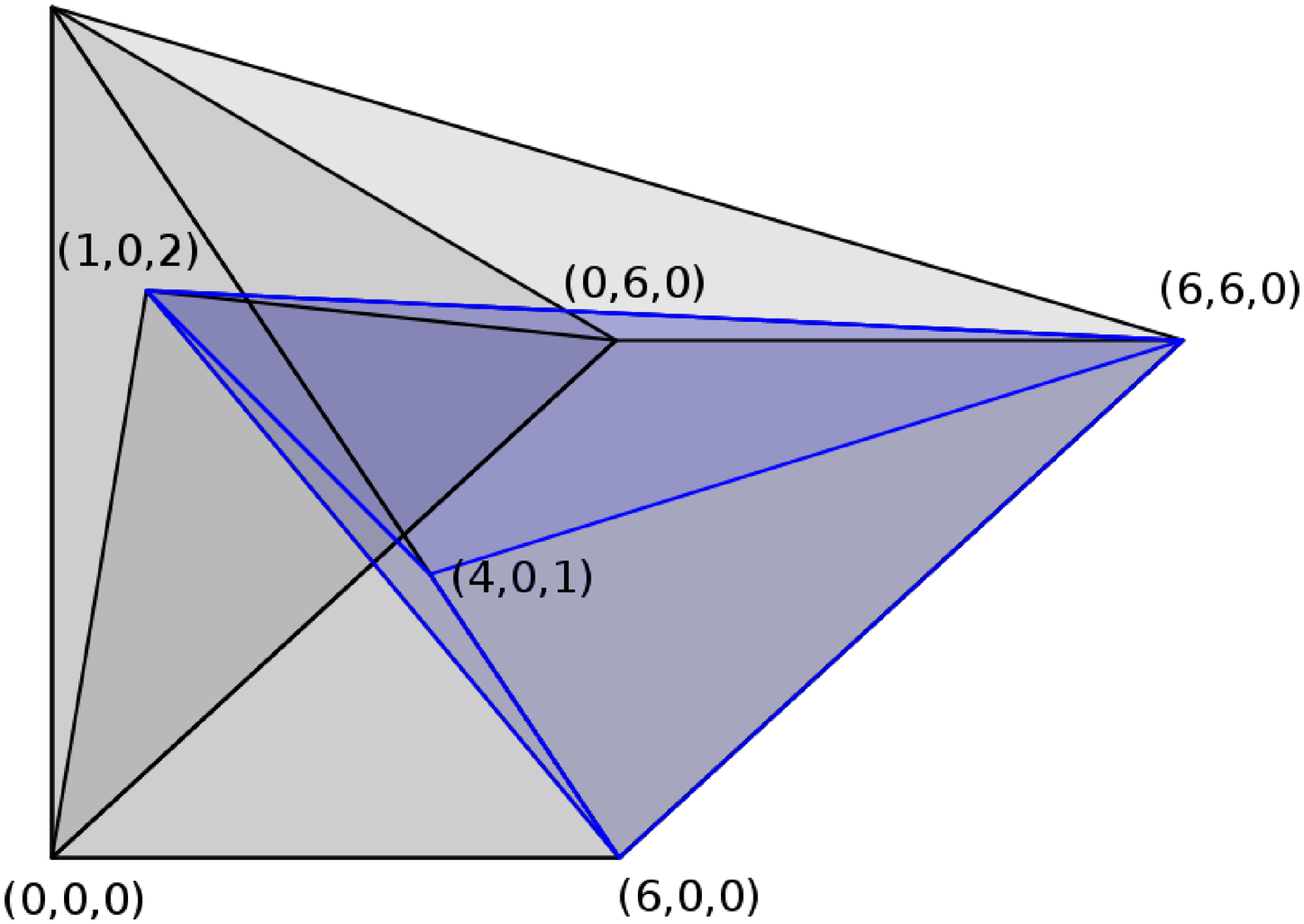}
\setlength\abovecaptionskip{-1cm}
\caption{Tetrahedron $\alpha_1$.}
\label{alpha1}
\end{figure}
Triangulate $\alpha_1$ into the cones with vertex $(4,0,1)$ over the triangles in the triangulation of the triangle with vertices $(1,0,2),(6,6,0),(6,0,0)$. Triangulate $\alpha_2$ into the cones with vertex $(0,0,1)$ over the triangles in the triangulation of the triangle with vertices $(1,0,2),(0,6,0),(0,0,0)$. All the tetrahedra of the triangulations constructed are primitive.
\\Consider the tetrahedra $\beta_1$ and $\beta_2$ with vertices $(1,0,2),(6,6,0),(4,4,1),(4,0,1)$ and $(1,0,2),(0,6,0),(0,4,1),(0,0,1)$ respectively. See Figure \ref{beta1} for a picture of $\beta_1$.
\begin{figure}[h!]
\includegraphics[width=18cm,height=8cm]{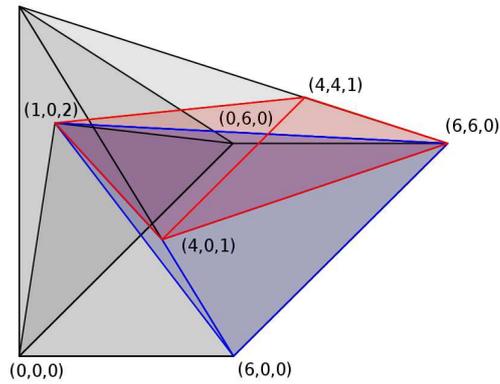}
\setlength\abovecaptionskip{-1cm}
\caption{Tetrahedron $\beta_1$.}
\label{beta1}
\end{figure}
Triangulate $\beta_1$ and $\beta_2$ into $4$ tetrahedra, respectively, using the subdivision of the segment $[(4,4,1)-(4,0,1)]$ and $[(0,4,1)-(0,0,1)]$ into four primitive edges. All the tetrahedra of the triangulations of $\beta_1$ and $\beta_2$ are primitive.
\\Consider the tetrahedron $\gamma_1$ with vertices $(1,0,2),(6,6,0),(4,4,1),(0,4,1)$, see Figure \ref{gamma1}.
\begin{figure}[h!]
\includegraphics[width=18cm,height=8cm]{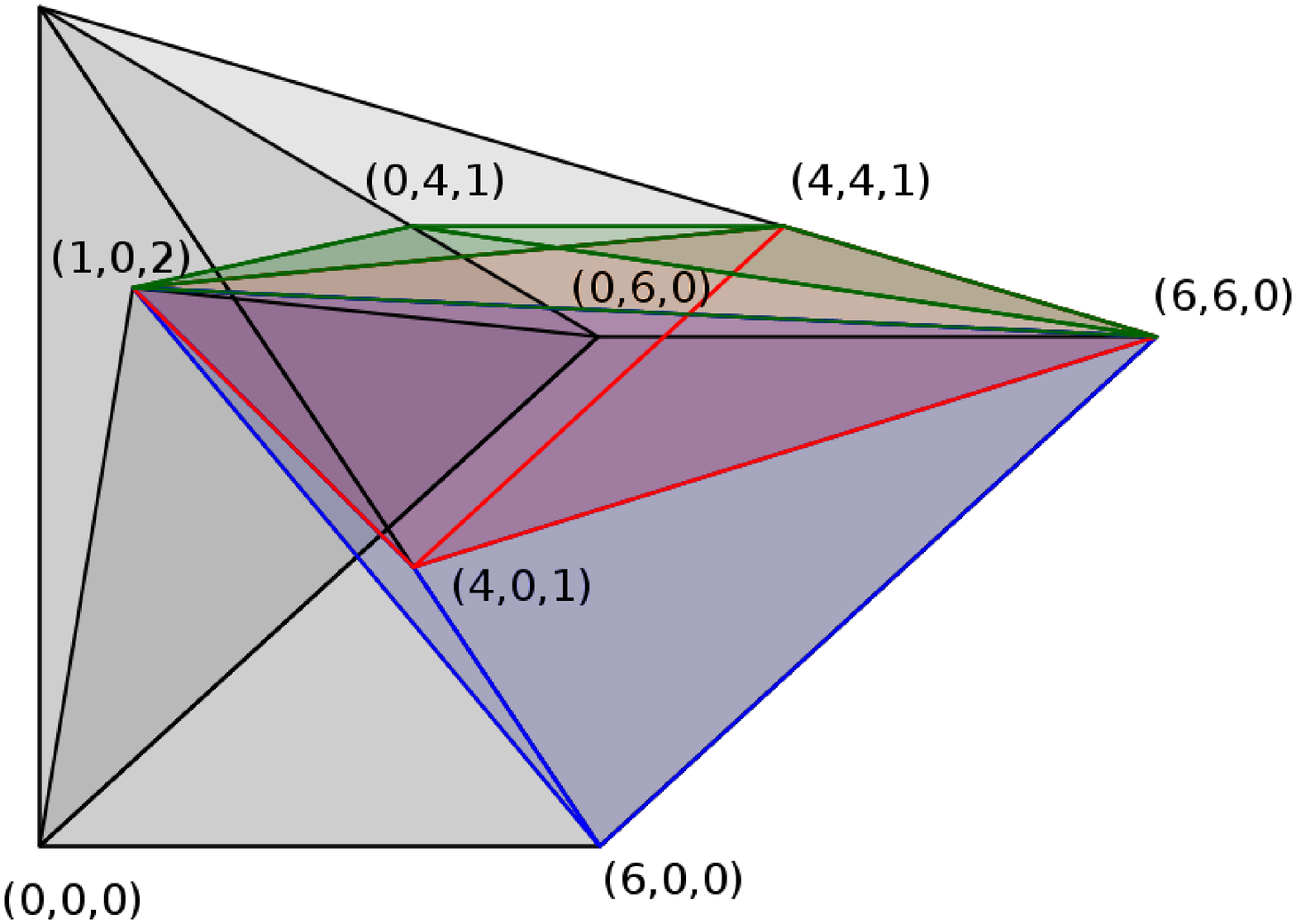}
\setlength\abovecaptionskip{-1cm}
\caption{Tetrahedron $\gamma_1$.}
\label{gamma1}
\end{figure}
Triangulate $\gamma_1$ into $4$ tetrahedra, using the subdivision of the segment $[(4,4,1)-(0,4,1)]$. All the tetrahedra of the triangulation of $\gamma_1$ are of volume $2$.
\\Consider the tetrahedron $\gamma_2$ with vertices $(1,0,2),(6,6,0),(0,6,0),(0,4,1)$. The triangle with vertices $(1,0,2), (6,6,0),(0,6,0)$ is already triangulated. Use this triangulation to subdivise $\gamma_2$. Finally, for the three edges $[(1,0,2)-(1,6,0)]$, $[(1,0,2)-(3,6,0)]$ and $[(1,0,2)-(5,6,0)]$ of length $2$, refine the triangulation as explained in Definition \ref{definition}.
\\At the present time, the part lying under the cone with vertex $(1,0,2)$ over $Q\cap\lbrace w=1\rbrace$ is triangulated (see Figure \ref{cone1}).
\begin{figure}[h!]
\includegraphics[width=18cm,height=8cm]{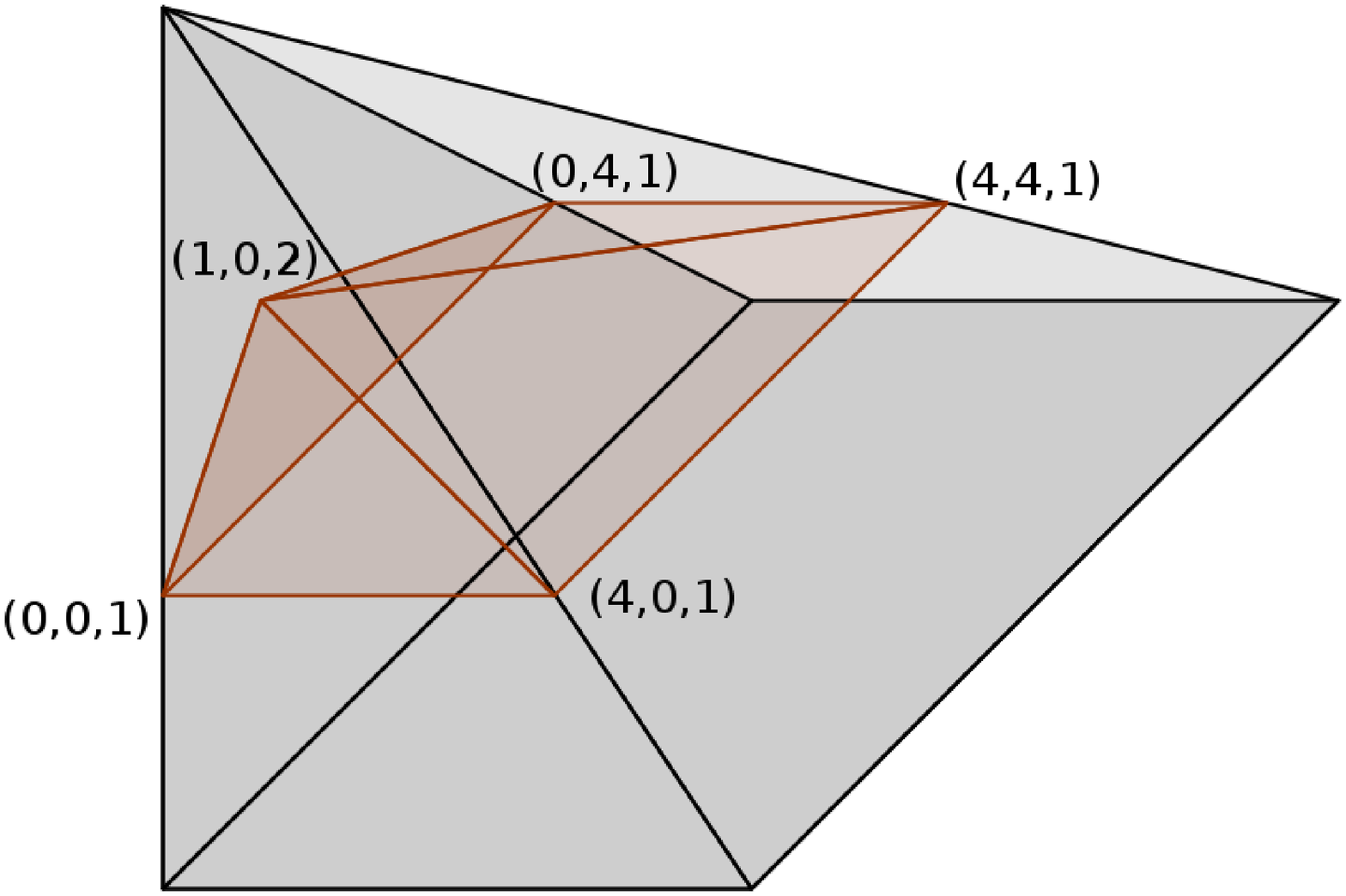}
\setlength\abovecaptionskip{-1cm}
\caption{Cone over $Q\cap\lbrace w=1\rbrace$.}
\label{cone1}
\end{figure}
Consider the pentagon $P$ with vertices $(1,0,2),(2,0,2),(2,2,2),(1,2,2),(0,1,2)$, triangulate it with any primitive convex triangulation and consider the two cones over it with vertex $(0,0,3)$ and $(4,4,1)$ respectively (see Figure \ref{pentagonS}).
\begin{figure}[h!]
\includegraphics[width=18cm,height=8cm]{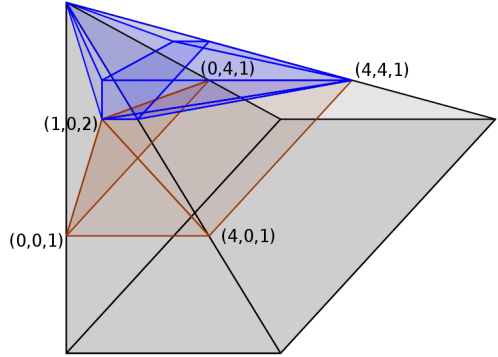}
\setlength\abovecaptionskip{-1cm}
\caption{Cones over the pentagon $P$.}
\label{pentagonS}
\end{figure}
Complete the triangulation considering the following tetrahedra:
\begin{itemize}
\item The joint of the segment $[(4,0,1)-(4,4,1)]$ and $[(1,0,2)-(2,0,2)]$ triangulated into $4$ primitive tetrahedra, using the triangulation of the segment $[(4,0,1)-(4,4,1)]$ into $4$ edges.
\item The joint of the segment $[(0,4,1)-(4,4,1)]$ and $[(0,1,2)-(0,2,2)]$ triangulated into $4$ primitive tetrahedra, using the triangulation of the segment $[(0,4,1)-(4,4,1)]$ into $4$ edges.
\item The joint of the segment $[(0,4,1)-(4,4,1)]$ and $[(1,0,2)-(0,1,2)]$ triangulated into $4$ primitive tetrahedra, using the triangulation of the segment $[(0,4,1)-(4,4,1)]$ into $4$ edges.
\item The two cones over the triangle $(0,0,2),(1,0,2),(0,1,2)$ with vertices $(0,0,1)$ and $(0,0,3)$, respectively.
\item The two cones over the triangle $(0,1,2),(0,2,2),(1,2,2)$ with vertices $(0,4,1)$ and $(0,0,3)$, respectively.
\end{itemize}
Denote by $\rho$ the obtained subdivision of $Q$. To show the convexity of $\rho$, one can proceed as in ~\cite{Itenberg1997}. First, remark that the ``coarse'' subdivision given by the cone $C$, the tetrahedra $\alpha_i$, the tetrahedra $\beta_i$, the tetrahedra $\gamma_i$, the cones over the pentagon $S$ and the remaining three joints and two cones is convex. Denote by $\nu^{\prime}$ a convex piecewise-linear function certifying the convexity of this ``coarse" subdivision. 
\\Choose three convex functions $\nu_1$, $\nu_2$ and $\nu_3$ certifying the convexity of the subdivision of the three edges $\left[ (0,0,1)-(0,4,1)\right] $, $\left[ (0,4,1)-(4,4,1)\right] $ and \linebreak $\left[ (4,4,1)-(4,0,1)\right] $. Choose also a convex function $\nu_4$ certifying the convexity of the chosen subdivision of the pentagon and a convex function $\nu_5$ certifying the convexity of the chosen subdivision of the cone $C$. 
\\Consider a piecewise-linear function $\nu:Q\rightarrow\R$ which is affine-linear on each tetrahedron of the subdivision $\rho$ and takes the value $\nu^\prime(x)+\sum\epsilon_i\nu_i(x)$ at every vertex $x$. The function $\nu$ for positive sufficiently small $\epsilon_i$ certifies the convexity of the subdivision $\rho$.
\\Define the distribution of signs $D(\tau)$ at the vertices of $\tau$. For the points inside $Q_0$, take the distribution of signs shown in Figure \ref{courbeaux}. Denote by $A$ a T-curve in $\PP^1\times\PP^1$ obtained from the triangulation $\tau$ and the distribution $D(\tau)$ restricted to $Q_0$. The chart of $A$ is depicted in Figure \ref{brugalle} b).
The distribution of signs at the vertices of $\tau$ belonging to $Q\cap\lbrace w\geq 1\rbrace$ is summarized in Figure \ref{distrib}.
\begin{figure}[h!]
\centerline{\includegraphics[width=19cm,height=10cm]{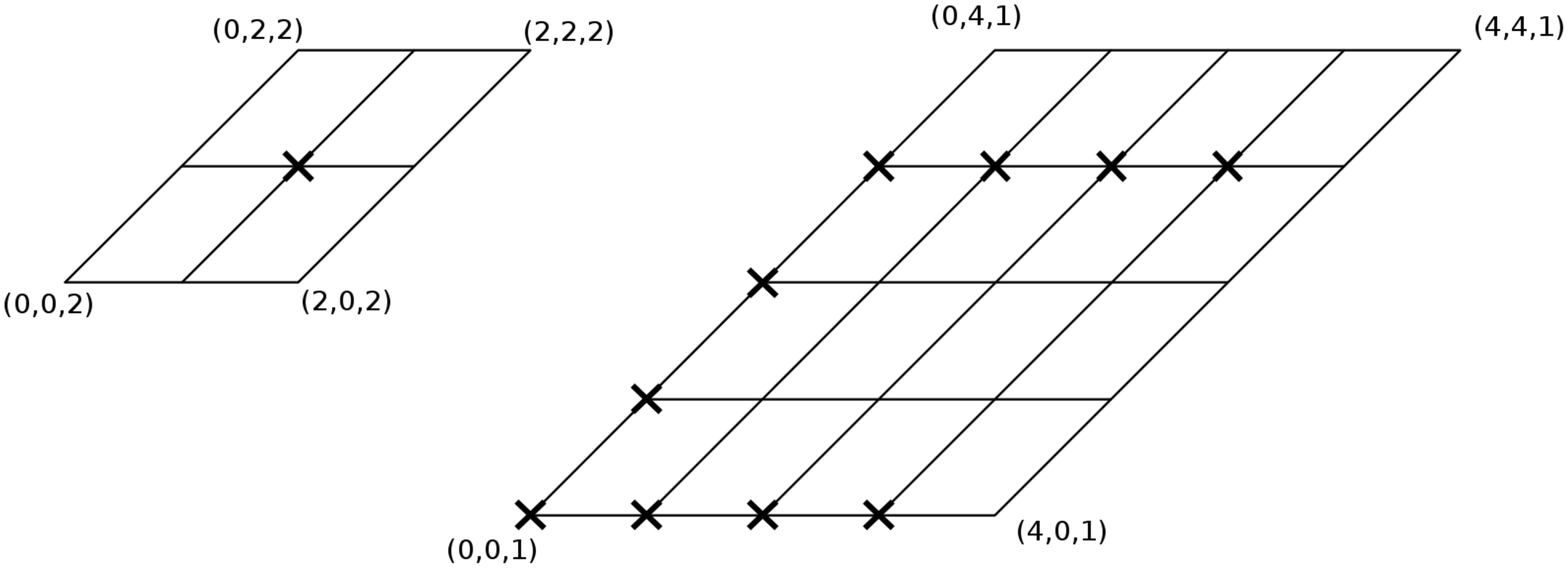}}
\setlength\abovecaptionskip{-3cm}
\caption{Distibution of signs for $w=1$ and $w=2$. A point gets a sign $+$ if and only if it is ticked.}
\label{distrib}
\end{figure}
The point $(0,0,3)$ gets the sign $+$.
\\Let us compute the Euler characteristic $\chi(\R Z)$ of $\R Z$. The triangulation $\tau$ contains $6$ edges of length $2$ with endpoints of opposite signs, and some tetrahedra of volume $2$ in $\gamma_1$ and in the cone $C$. Since all the other tetrahedra are elementary and the stars of the four edges of length $2$ are disjoint, we can use Propositions \ref{prop2} and \ref{prop3} to compute $\chi(\R Z)$. In $\gamma_1$ all the signs are positive, and in the cone $C$, six tetrahedra of volume $2$ have negative product of signs. One obtains:
$$
\chi(\R Z)=\sigma(\C Z)+12=-52.
$$
\subsection{The surface $Y$}
To construct the surface $Y$, we use a real trigonal curve $(C_3=0)$ constructed by E. Brugallé in  ~\cite{Brugalle2006}. The Newton polygon of the polynomial $C_3$ is \\$Conv((0,0),(6,0),(0,3),(6,1))$ and the chart of $C_3$ is depicted in Figure \ref{c3}.
\begin{figure}[h!]
\centerline{
\includegraphics[width=9cm,height=5cm]{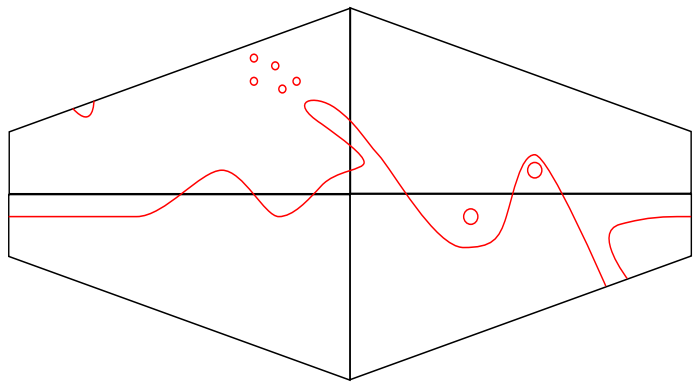}}
\setlength\abovecaptionskip{0cm}
\caption{Chart of $(C_3=0)$.}
\label{c3}
\end{figure}
\\Denote by $\Gamma$ the hexagon $Conv((0,0),(4,1),(6,2),(6,4),(4,5),(0,6))$. Consider the charts of the polynomials
\begin{itemize}
\item $Y^3C_3(X,Y)$, $Y^3C_3(X,\frac{1}{Y})$,
\item $Y^6b(X^3\frac{1}{Y},X^4\frac{1}{Y})$, $b(X^3Y,X^4Y)$,
\end{itemize} 
where $b(X,Y)=Y+(X+x_1)(X+x_2)$, with $x_1,x_2>0$ appropriately chosen so that the restrictions of the polynomials $C_3(X,Y)$ and $Y^3b(X^3\frac{1}{Y},X^4\frac{1}{Y})$ to $Conv((0,3);(6,1))$ are equal.
\begin{figure}[h!]
\centerline{
\includegraphics[width=7cm,height=7cm]{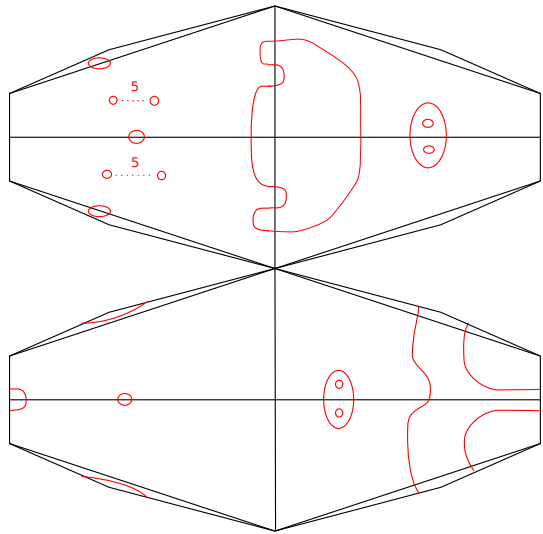}}
\setlength\abovecaptionskip{0cm}
\caption{Chart of the polynomial P.}
\label{test3}
\end{figure}
By Viro's patchworking theorem, there exists a polynomial $P$ of Newton polygon $\Gamma$ whose chart is depicted in Figure \ref{test3}. 
To construct the surface $Y$, apply the general Viro's patchworking inside $Q$ with 
\begin{itemize}
\item the chart of $xz^2+P(x,y)$ inside $Conv(\Gamma,(1,0,2))$,
\item the same triangulation and distribution of signs as in Section \ref{aux} outside $Conv(\Gamma,(1,0,2))$.
\end{itemize}

 Denote by $\widehat{A}$ the curve in $\PP^1\times\PP^1$ obtained as the intersection of $Y$ with the toric divisor corresponding to the face $Q_0$. See Figure \ref{brugalle} a).
\begin{figure}[h!]
\begin{minipage}[c]{0.6\linewidth}
\includegraphics[width=6cm,height=7cm]{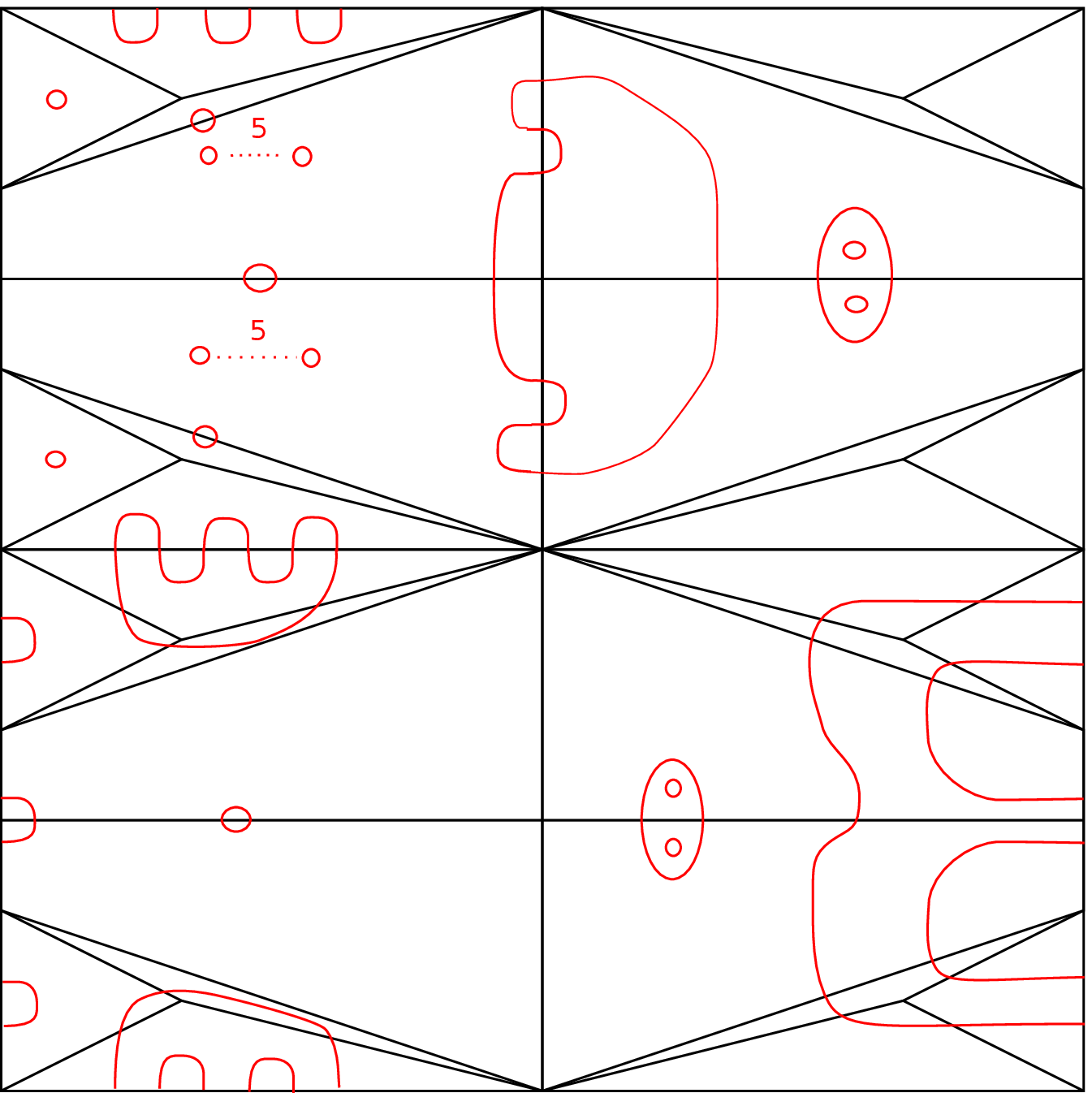} 
\begin{center}
\textbf{a)}
\end{center}
\end{minipage}
\begin{minipage}[c]{0.6\linewidth}
\includegraphics[width=6cm,height=7cm]{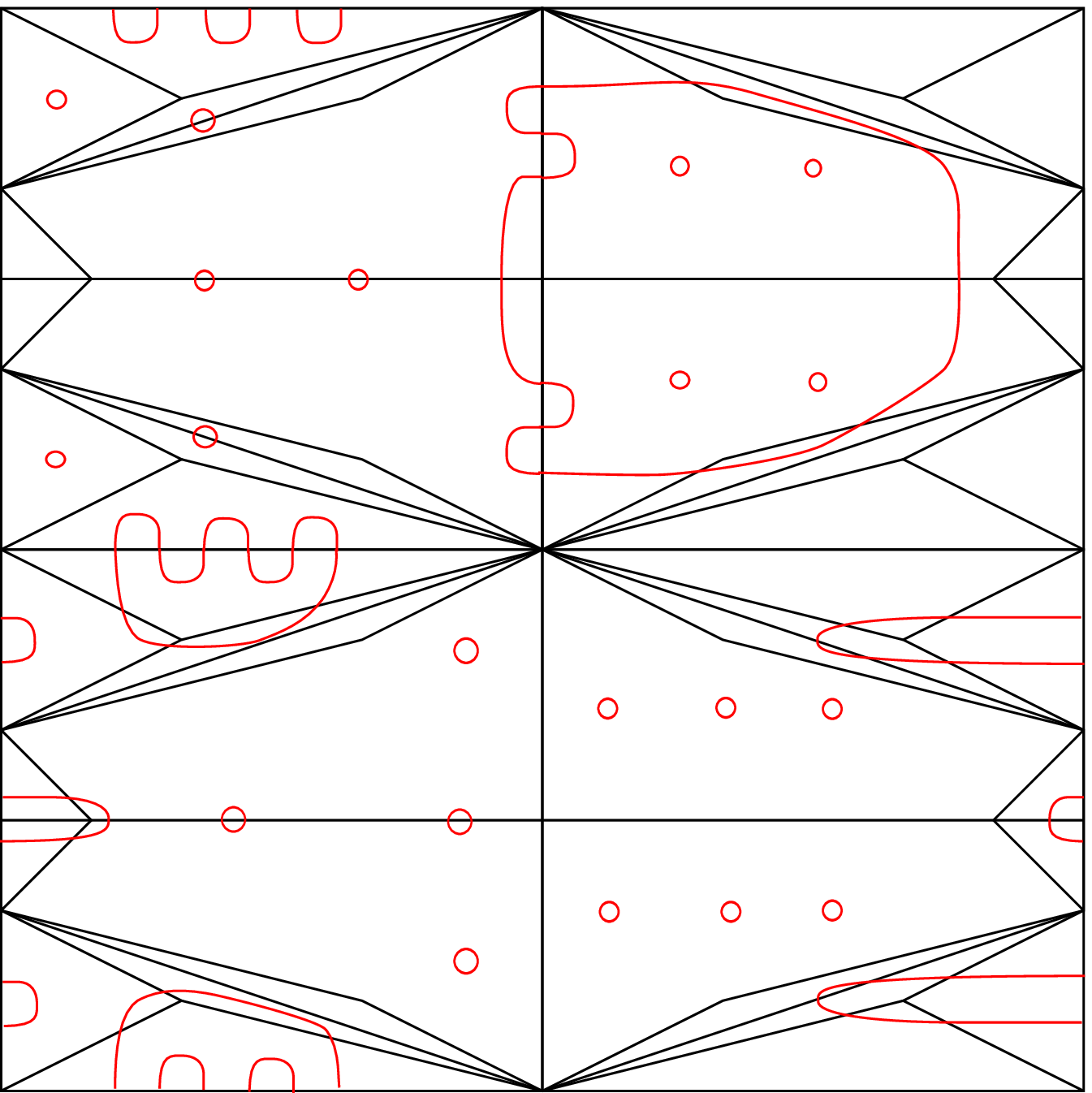}
\begin{center}
\textbf{b)}
\end{center}
\end{minipage}

\setlength\abovecaptionskip{1cm}
\caption{a): $\R \widehat{A}$ \ \ \ \ \ b): $\R A$}
\label{brugalle}
\end{figure}

Let us now compute the Euler caracteristic of $\R Y$. To compute it, we compare the Euler caracteristics of $\R Z$ and $\R Y$. 
First of all, denote $Z_1$ (resp., $Y_1$) the surfaces constructed in the same way as $Z$ (resp., $Y$) but where the six edges $[(1,0,2)-(1,0,0)]$, $[(1,0,2)-(3,0,0)]$, $[(1,0,2)-(5,0,0)]$, $[(1,0,2)-(1,6,0)]$, $[(1,0,2)-(3,6,0)]$ and $[(1,0,2)-(5,6,0)]$ are not refined. From Proposition \ref{prop3}, one obtains
$$
\chi(\R Y)-\chi(\R Z)=\chi(\R Y_1)-\chi(\R Z_1).
$$
Then, notice that outside of $C$, the triangulation and distribution of signs defining $Z_1$ and $Y_1$ coincide. Denote by $Z_2$ (resp., $Y_2$) the surfaces with Newton polygon $C$, defined by $(A(x,y)+xz^2=0)$ (resp., $(\widehat{A}(x,y)+xz^2=0)$) and compactified in $\mathit{Tor}(C)$. These surfaces are singular, with $12$ ordinary double points. 
However, there exist two homeomorphic compact sets $B\subset\R\mathit{Tor}(Q)$ and $B'\subset\R\mathit{Tor}(C)$ such that:
\begin{itemize} 
\item $\R Y_1\setminus B$ is homeomorphic to $\R Z_1\setminus B$,
\item $\R Y_2\setminus B'$ is homeomorphic to $\R Z_2\setminus B'$,
\item $\R Y_1\cap B$ is homeomorphic to $\R Y_2\cap B'$,
\item $\R Z_1\cap B$ is homeomorphic to $\R Z_2\cap B'$.
\end{itemize}
So one has:
$$
\chi(\R Y_2\cap B')-\chi(\R Z_2\cap B')=\chi(\R Y_1\cap B)-\chi(\R Z_1\cap B).
$$
By the additivity of the Euler caracteristic, one also has that
$$
\chi(\R Y_1\cap B)-\chi(\R Z_1\cap B)=\chi(\R Y_1)-\chi(\R Z_1),
$$
and
$$
\chi(\R Y_2\cap B')-\chi(\R Z_2\cap B')=\chi(\R Y_2)-\chi(\R Z_2).
$$
So finally
$$
\chi(\R Y_2)-\chi(\R Z_2)=\chi(\R Y_1)-\chi(\R Z_1).
$$
It remains to compute $\chi(\R Y_2)$ and $\chi(\R Z_2)$.
Topologically, $\R Z_2$ is obtained by taking in the quadrant $++$ and $+-$ (resp., $-+$ and $--$) the ``double" of $(A\leq 0)$ (resp., $(A\geq 0)$) ramified along $(A=0)\cup(x=0)\cup(x=\infty)$. The same holds for $\R Y_2$ by replacing $A$ with $\widehat{A}$, see Figure \ref{double}. 
\begin{figure}[h!]
\begin{minipage}[l]{0.6\linewidth}
\includegraphics[width=6cm,height=7cm]{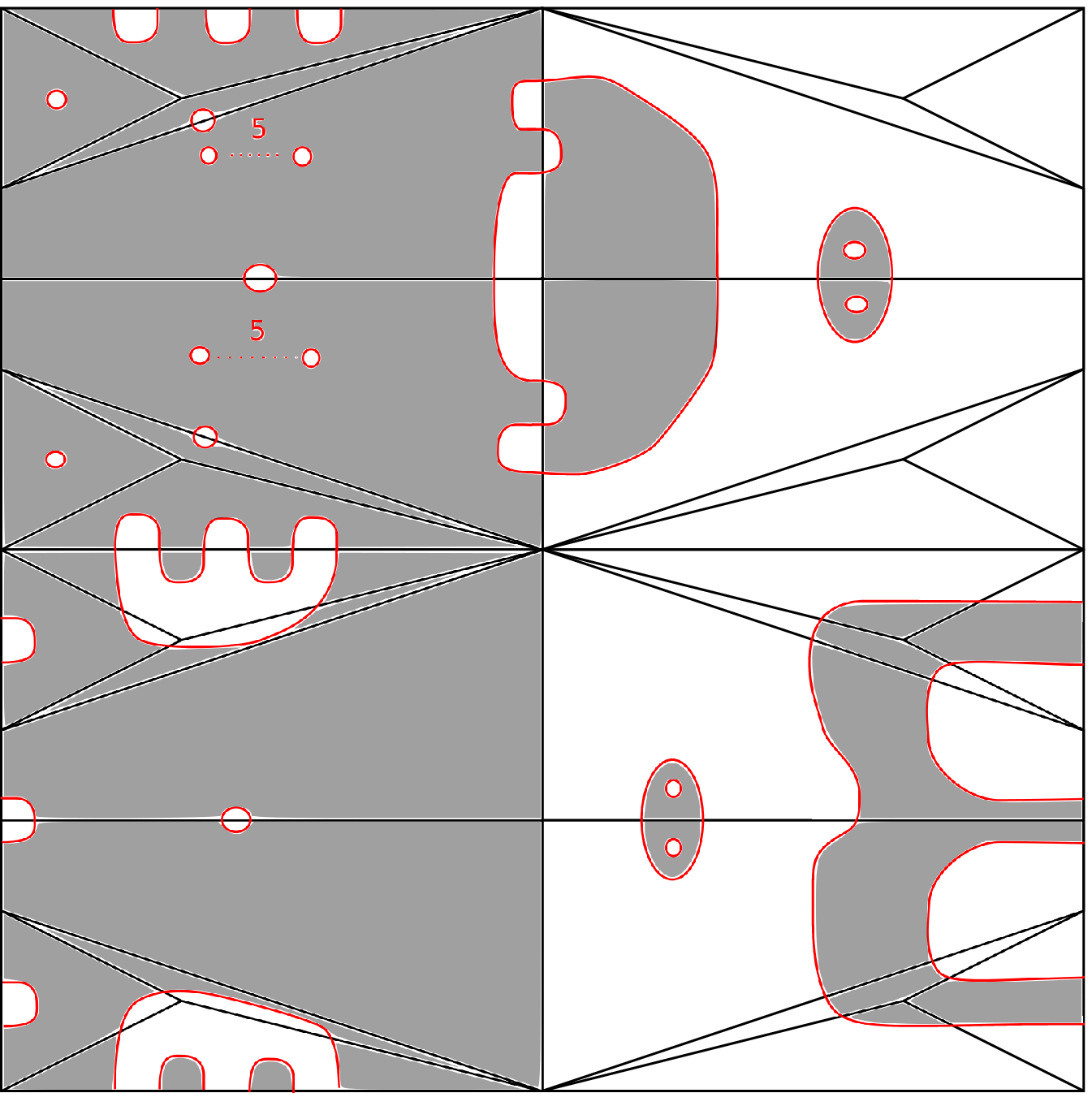} 
\begin{center}
\textbf{a)}
\end{center}
\end{minipage}
\begin{minipage}[r]{0.6\linewidth}
\includegraphics[width=6cm,height=7cm]{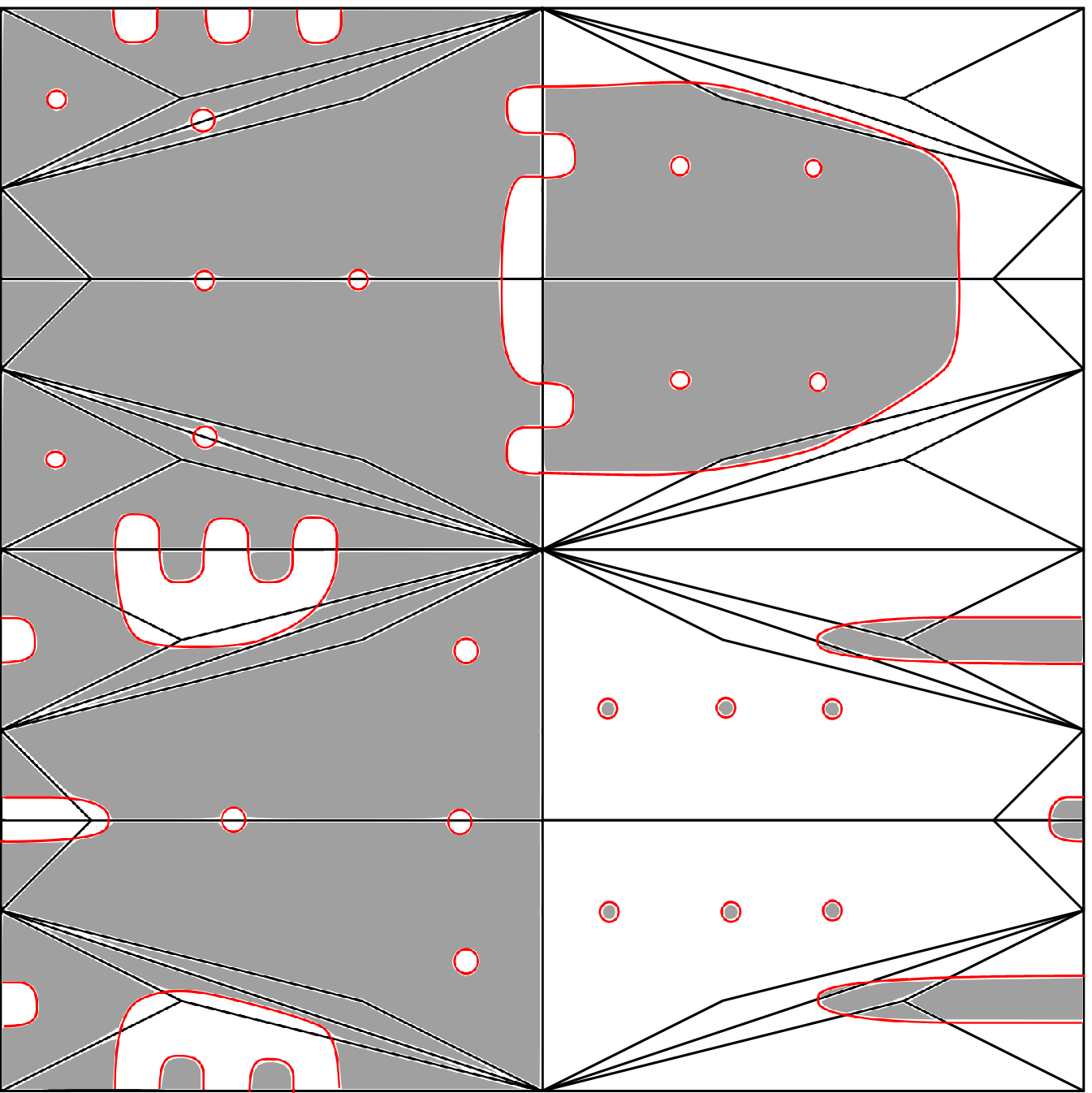}
\begin{center}
\textbf{b)}
\end{center}
\end{minipage}

\setlength\abovecaptionskip{1cm}
\caption{a): $(A(x,y)x<0)$  \ \ \ \ \  b): $(\widehat{A}(x,y)x<0)$}
\label{double}
\end{figure}
By a direct computation, we obtain
$$
\chi(\R Y_2)=2(-18)-12=-48,
$$
and 
$$
\chi(\R Z_2)=2(-6)-12=-24.
$$
Then,
$$
\chi(\R Y)-\chi(\R Z)=\chi(\R Y_2)-\chi(\R Z_2)=-24.
$$
So finally
$$
\chi(\R Y)=\chi(\R Z)-24=-52-24=-76.
$$

Moreover, $\R Y$ contains two components homeomorphic to $S_2$ coming from the double covering of $(\widehat{A}>0)$.
Note that the vertices $(1,1,2),(1,3,1),(2,3,1)$ and $(3,3,1)$ have the following property: all the vertices of the triangulation connected to one of these vertices by an edge have the sign $+$, while the vertices $(1,1,2),(1,3,1),(2,3,1)$ and $(3,3,1)$ have the sign $-$. Thus, $\R Y$ contains also four spheres. 
There is at least one component of $\R Y$ more: this component intersects the plane $\lbrace u=0\rbrace$. Moreover, $\R Y$ cannot have more components, otherwise $Y$ would be an $M$-surface, but $\chi(\R Y)$ does not satisfy the Rokhlin congruence. 
Finally, from $\chi(\R Y)=-76$, we obtain
$$
\R Y\simeq 4S\amalg 2S_2\amalg  S_{41}.
$$

\clearpage
\bibliographystyle{alpha}
\bibliography{biblio}
\end{document}